\documentclass[11pt,reqno]{amsart}
\usepackage{amssymb,latexsym}
\usepackage{graphicx}
\usepackage{mathtools}
\usepackage{color}
\usepackage[T1]{fontenc}
\usepackage{hyperref}
\usepackage{amsmath}
\usepackage{breqn}
\usepackage[toc,page]{appendix}
\usepackage{url}    
\usepackage{cancel}
\usepackage{multirow}
\newcommand{\bburl}[1]{\textcolor{blue}{\url{#1}}}

\calclayout

\makeatletter
\newcommand{\monthyear}[1]{%
  \def\@monthyear{\uppercase{#1}}}
\newcommand{\volnumber}[1]{%
  \def\@volnumber{\uppercase{#1}}}
\AtBeginDocument{%
\def\ps@plain{\ps@empty
  \def\@oddfoot{\@monthyear \hfil \thepage}%
  \def\@evenfoot{\thepage \hfil \@volnumber}}
\def\ps@firstpage{\ps@plain}
\def\ps@headings{\ps@empty
  \def\@evenhead{%
    \setTrue{runhead}%
    \def\thanks{\protect\thanks@warning}%
    \uppercase{\ }\hfil}%
  \def\@oddhead{%
    \setTrue{runhead}%
    \def\thanks{\protect\thanks@warning}%
    \hfill\uppercase{Winning Strategy for Multiplayer and Multialliance Zeckendorf Game}}%
  \def\@evenhead{%
    \setTrue{runhead}%
    \def\thanks{\protect\thanks@warning}%
    \uppercase{The Fibonacci Quarterly}\hfil}%
  \let\@mkboth\markboth
  \def\@evenfoot{%
    \thepage \hfil \@volnumber}%
  \def\@oddfoot{%
    \@monthyear \hfil \thepage}%
  }%
\footskip=25pt
\pagestyle{headings}%
}
\makeatother

\theoremstyle{plain}
\numberwithin{equation}{section}
\newtheorem{thm}{Theorem}[section]

\newtheorem{property}{Property}

\newcommand{\ignore}[1]{}

\newcommand\be{\begin{eqnarray}}
\newcommand\ee{\end{eqnarray}}
\newcommand\bea{\begin{eqnarray}}
\newcommand\eea{\end{eqnarray}}
\newcommand\ben{\begin{enumerate}}
\newcommand\een{\end{enumerate}}

\newtheorem{lem}[thm]{Lemma}

\newtheorem{defi}[thm]{Definition}

\begin{document}

\monthyear{TBD}
\volnumber{Volume, Number}
\setcounter{page}{1}

\title{Winning Strategy for Multiplayer and Multialliance Zeckendorf Game}

\author[]{Anna Cusenza}
\address{University of California, Los Angeles, Los Angeles, CA 90095}
\email{ascusenza@g.ucla.edu}

\author[]{Aidan Dunkelberg}
\address{Williams College, Williamstown, MA 01267}
\email{awd4@williams.edu}

\author[]{Kate Huffman}
\address{University of Alabama, Tuscaloosa, AL 35401}
\email{klhuffman@crimson.ua.edu}

\author[]{Dianhui Ke}
\address{University of Michigan, Ann Arbor, MI}
\email{kdianhui@umich.edu}

\author[]{Daniel Kleber}
\address{Carleton College, 300 North College Street, Northfield, MN 55057}
\email{kleberd@carleton.edu}

\author[]{\\Steven J. Miller}
\address{Department of Mathematics and Statistics, Williams College, Williamstown, MA 01267}
\email{sjm1@williams.edu,  Steven.Miller.MC.96@aya.yale.edu}

\author[]{Clayton Mizgerd}
\address{Department of Mathematics and Statistics, Williams College, Williamstown, MA 01267}
\email{cmm12@williams.edu}

\author[]{Vashisth Tiwari}
\address{University of Rochester, Rochester, NY 14627}
\email{vtiwari2@u.rochester.edu}

\author[]{Jingkai Ye}
\address{Whitman College, 280 Boyer Avenue, Walla Walla, WA, 99362}
\email{yej@whitman.edu.edu}

\author[]{Xiaoyan Zheng}
\address{Washington University in St. Louis, St. Louis, MO 63130}
\email{zhengxiaoyan@wustl.edu}

\date{\today}

\begin{abstract}
Edouard Zeckendorf proved that every positive integer $n$ can be uniquely written \cite{Ze} as the sum of non-adjacent Fibonacci numbers, known as the Zeckendorf decomposition. Based on Zeckendorf's decomposition, we have the Zeckendorf game for multiple players.
We show that when the Zeckendorf game has at least $3$ players, none of the players have a winning strategy for $n\geq 5$. Then we extend the multi-player game to the multi-alliance game, finding some interesting situations in which no alliance has a winning strategy. This includes the two-alliance game, and some cases in which one alliance always has a winning strategy.


\end{abstract}

\maketitle

\ \\


\section{Introduction}

\subsection{Rules of Zeckendorf Game}
The Fibonacci sequence is one of the most fabulous sequences with a number of beautiful properties. Among these properties is a theorem by Edouard Zeckendorf \cite{Ze}. Zeckendorf proved that every positive integer $n$ can be uniquely written as the sum of distinct non-consecutive Fibonacci numbers. This sum is also known as the \textit{Zeckendorf decomposition} of $n$. We define the Fibonacci sequence as $F_1 = 1, F_2 = 2, F_3 = 3$, and $F_{n+1}=F_n+F_{n-1}$. If we stuck with $F_1=F_2=1$, we lose uniqueness.

Baird-Smith, Epstein, Flint and Miller \cite{BEFMgen,BEFM} create a game based on the Zeckendorf decomposition. We quote from \cite{BEFM} to describe the game.

We first introduce some notation. By $\{1^n\}$ or $\{{F_1}^n\}$ we mean $n$ copies of $1$, or $F_1$. If we have $3$ copies of $F_1$, $2$ copies of $F_2$, and $7$ copies of $F_4$, we write either $\{{F_1}^3 + {F_2}^2 + {F_4}^7 \}$ or $\{1^3 + 2^2 + 5^7\}$.
\begin{defi}[The Zeckendorf Game]\label{defi:zg}
At the beginning of the game, there is an unordered list of $n$ 1's. Let $F_1 = 1, F_2 = 2$, and $F_{i+1} = F_i + F_{i-1}$; therefore the initial list is $\{{F_1}^n\}$. On each turn, a player can do one of the following moves.
\begin{enumerate}
\item If the list contains two consecutive Fibonacci numbers, $F_{i-1}, F_i$, these can be combined to create $F_{i+1}$. We denote this move by $F_{i-1} + F_i = F_{i+1}$.
\item If the list has two of the same Fibonacci number, $F_i, F_i$, then
\begin{enumerate}
\item if $i=1$, $F_1, F_1$ can be combined to create $F_2$, denoted by $1+1=2$,
\item if $i=2$, a player can change $F_2, F_2$ to $F_1, F_3$, denoted by $2+2=1+3$, and
\item if $i \geq 3$, a player can change $F_i, F_i$ to $F_{i-2}, F_{i+1}$, denoted by $F_i + F_i = F_{i-2}+ F_{i+1}$.
\end{enumerate}
\end{enumerate}
The players alternate moving. The game ends when one player moves to create the Zeckendorf decomposition.
\end{defi}

In the following results, $p$ represents the number of players in the Zeckendorf game.

\subsection{Previous Results} Baird-Smith, Epstein, Flint and Miller \cite{BEFM} proved the following results in the Zeckendorf game.

\begin{thm}
Every game terminates in a finite number of moves at the Zeckendorf decomposition.
\end{thm}

\begin{thm}\label{thm1}
In the two-player game $(p=2)$, for any $n>2$, player $2$ always has a winning strategy.
\end{thm}

It is worth noting that the proof on Theorem \ref{thm1} is non-constructive, and it is still an open problem to find the constructive winning strategy of player $2$.

\subsection{New Results}

\begin{thm}\label{thm2} When $n \geq 5$, for any $p \geq 3$ no player has a winning strategy.
\end{thm}

Now we extend the multi-player games to the game of more than two alliances, or teams. We use $t$ to represent the number of teams. We then have the following theorems.

\begin{thm}\label{thm3}
For any $n\geq2k^2+4k$ and $t \geq 3$, if each team has exactly $k=t-1$ consecutive players, then no team has a winning strategy.
\end{thm}


\begin{thm}\label{thm4}
Let $n\geq 30$ and $p=6$. If one alliance has $4$ players and the other alliance has $2$ players, then the $4$-player alliance always has a winning strategy.
\end{thm}

\begin{thm}\label{thm5}
Let $n\geq 32$ and $p\geq 7$. If there are two alliances, one alliance with $p-2$ players (called the big alliance), and the other with exactly $2$ players (called the small alliance), then the big alliance always has a winning strategy.
\end{thm}

Lastly, we extend this to larger alliances with the following theorems.


\begin{thm}\label{thm6}
For any $n\geq4pb+2p-2b$, if one alliance contains more than two-thirds of the players and there is some integer $b$ such that if a player $i$ is not on the alliance, player $(i - b)\bmod p$ is on it, and the alliance has at least $2b$ players in a row, then that alliance has a winning strategy.
\end{thm}

\begin{thm}\label{thm7}
Let $n\geq 2p+4b$. If there is some integer $b$ such that if a player $i$ is not on the alliance then player $(i - b)\bmod p$ is on the alliance, and the alliance has at least $3b$ players in a row, then that alliance has a winning strategy.
\end{thm}


\begin{thm}\label{thm8}
Assume we have one big alliance (sized $2d$) and one small alliance (sized $d$). In this two-alliance game consisting of $3d$ players ($d$ can be any positive integer), when the small alliance consists of $d$ consecutive players and the big alliance consists of $2d$ consecutive players, if $n\geq 12d^2+4d$, then the big alliance always has a winning strategy.
\end{thm}

\section{Winning Strategy for Zeckendorf Game}

\subsection{Proof Of Theorem \ref{thm2}}

Note: In all the following proofs of this section, player $0 =$ player $p$ $($under $\bmod$ $p)$.

To prove Theorem \ref{thm2}, we introduce the following property:
\begin{property}\label{p1}
Suppose player $m$ has a winning strategy $(1\leq m \leq p)$. For any $p\geq 3$, if player $m$ is not the player who takes step $2$ listed below, then any winning path of player $m$ does not contain the following $3$ consecutive steps:

Step $1: 1+1=2$.

Step $2: 1+1=2$.

Step $3: 2+2=1+3$.
\end{property}

\begin{proof}
Suppose player $m$ has a winning strategy and there is a winning path that contains these $3$ consecutive steps. Then there exists a player $a$ where $1\leq a\leq p,\ a \neq m$, such that player $a-1 \pmod p$ can take step $1$, player $a$ can take step $2$ and player $a+1 \pmod p$ can take step $3$.

Note that instead of doing $1+1=2$, player $a$ can do $1+2=3$. Then player $m-1 \pmod p$ has a winning strategy, which is a contradiction.

Therefore, by using stealing strategy, Property \ref{p1} holds.
\end{proof}

We now prove Theorem \ref{thm2} by splitting it into the following $2$ lemmas, Lemma \ref{lem2.1.1} and Lemma \ref{lem2.1.2}.

\begin{lem}\label{lem2.1.1}
When $n \geq 13$, for any $p \geq 4$ no player has a winning strategy.
\end{lem}
\begin{proof}
Suppose player $m$ has a winning strategy $(1\leq m\leq p)$.

Consider the following two cases.\\

\begin{itemize}
\item[Case 1.] If $m\geq 4$, then player $1, 2, 3$ can do the following:

Player $1: 1+1=2$.

Player $2: 1+1=2$.

Player $3: 2+2=1+3$.

This contradicts Property \ref{p1}, so player $m$ does not have winning strategy for any $m\geq 4$.\\

\item[Case 2.] If $m\leq 3$, then after player $m$'s first move, player $m+1, m+2, m+3$ can do the following:

Player $m+1: 1+1=2$.

Player $m+2: 1+1=2$.

Player $m+3: 2+2=1+3$.

This contradicts Property \ref{p1}, so player $m$ does not have winning strategy for any $m\leq 3$.
\end{itemize}

By Case $1$ and Case $2$, Lemma \ref{lem2.1.1} is proved.
\end{proof}

\begin{lem}\label{lem2.1.2}
When $n \geq 13$, for $p=3$ no player has a winning strategy.
\end{lem}

\begin{proof}
Suppose player $m$ has a winning strategy $(1\leq m\leq 3)$.

After player $m$'s first move, players $m+1$ and $m+2$ can do the following $($if $m=3$, we can start the following process from the first step of the game$)$:

Player $m+1: 1+1=2$ (Step $1$).

Player $m+2: 1+1=2$ (Step $2$).

Player $m$: player $m$ can do anything (Step $3$).\\


Note that if player $m$ does $2+2=1+3$, then Steps $1, 2$, and $3$ violate Property \ref{p1}, which is a contradiction.

 If player $m$ does anything else other than $2+2=1+3$, then after player $m$'s first move, the other $2$ players can do the following (continuing after the first $3$ steps listed above with $2$ more steps; if $m=3$, player $m+1$ is player $1$):

Player $m+1: 1+1=2$ (Step $1$).

Player $m+2: 1+1=2$ (Step $2$).

Player $m$: player $m$ can do anything (Step $3$).

Player $m+1: 1+1=2$ (Step $4$).

Player $m+2: 2+2=1+3$ (Step $5$).

Note that Step $3$ removes at most one $2$, but Step $1$ and Step $2$ generate two $2$'s in total, so there will be at least one $2$ remaining after step $3$. Therefore, player $m+1$ can do $1+2=3$ instead in Step $4$. By doing so, now player $m-1\pmod p$ has winning strategy, which is a contradiction.

Thus by using a stealing strategy, Lemma \ref{lem2.1.2} is proved.
\end{proof}

By Lemmas \ref{lem2.1.1} and \ref{lem2.1.2}, and brute force computations for $n<13$, Theorem \ref{thm2} is proved.

\subsection{Proof Of Theorem \ref{thm3}}

For the following proofs, team $0 =$ team $t$ $($under $\bmod\ t)$.

Note that player $tk$'s next player is player $1$, and we regard player $tk$ and player $1$ as two consecutive players. Therefore, without loss of generality, in all the following proofs, we assume that team $1$ has player $1, 2, 3,\dots, k;$ team $2$ has player $k+1, k+2,\dots, 2k;$ team $3$ has player $2k+1, 2k+2,\dots, 3k$ and so on.

For this proof we utilize the following property.

\begin{property}\label{p2}
Suppose team $m$ has a winning strategy $(1\leq m \leq t)$. For any $t\geq 3$ and $k=t-1$, if none of the middle $k$ players listed below belong to team $m$, then any winning path for team $m$ does not contain the following $3k$ consecutive steps:

First $k$ players all do: $1+1=2$.

Middle $k$ players all do: $1+1=2$.

Last $k$ players all do: $2+2=1+3$.
\end{property}

\begin{proof}
Suppose team $m$ has a winning strategy and there is a winning path for team $m$ that contains such $3k$ consecutive steps. Then $\exists\ q$ $(1\leq q\leq p)$ such that player $q$ belongs to team $m$ and takes the last step of the game.

For the middle $k$ players, instead of doing $1+1=2$, they can all do $1+2=3$.

By doing so, player $q-k$ now becomes the player who takes the last step.

Note that team $m$ has $k$ players, so player $q-k$ belongs to team $m-1 \pmod t$.

So team $m-1 ($ mod $t)$ has winning strategy, which contradicts with our assumption.

Therefore, by using stealing strategy, it is proved that Property \ref{p2} holds.
\end{proof}

After proving Property \ref{p2}, we prove Theorem \ref{thm3} by splitting it into the following $2$ lemmas: Lemmas \ref{lem3} and \ref{lem4}.

\begin{lem}\label{lem3}
When $n\geq2k^2+4k$, for any $t \geq 4$ and $k=t-1$ no team has winning strategy.
\end{lem}

\begin{proof}
Suppose team $m$ has a winning strategy $(1\leq m\leq t)$.

Note that the last player in team $m$ is player $mk$, so the first player after team $m$ is player $mk+1 \pmod p$.

Also, since there are $t-1=k$ other teams, and each team has $k$ players, where $k\geq 4-1=3$. Therefore, there are $k^2\geq 3k $ consecutive players from other teams. After all the members of team $m$'s first move, the consecutive $t-1=k$ other teams can do the following:

$($If $m=t$, we start the following steps from the first step of player $1.)$ In all the following, given players' numbers are mod $p$.

From player $mk+1$ to $(m+1)k$ all do: $1+1=2$.

From player $(m+1)k+1$ to player $(m+2)k$ all do: $1+1=2$.

From player $(m+2)k+1$ to player $(m+3)k$ all do: $2+2=1+3$.

Since all these $3k$ players are not from team $m$, it contradicts with property \ref{p2}, so team $m$ does not have winning strategy. Therefore, Lemma \ref{lem3} is proved.
\end{proof}

\begin{lem}\label{lem4}
When $n\geq 30$, for any $t=3$ and $k=2$ no team has winning strategy.
\end{lem}

\begin{proof}
Suppose team $m$ has a winning strategy $(1\leq m\leq 3)$. Note that the game has $3$ teams and $6$ players in total, so all the players' numbers listed below are under mod $6$, and all the teams' numbers listed below are under mod $3$. Team $m+1$ has players $2m+1$ and $2m+2$; team $m-1$ has players $2m+3$ and $2m+4$; team $m$ has players $2m-1$ and $2m$.

After player $2m$'s (last player from team $m$) first move, let's do the following first:

$($if $m=3$, the same following process can start from the first step of player $1.)$

Player $2m+1: 1+1=2$ $($Step $1)$.

Player $2m+2: 1+1=2$ $($Step $2)$.

Player $2m+3: 1+1=2$ $($Step $3)$.

Player $2m+4: 1+1=2$ $($Step $4)$.

Player $2m-1:$ anything $($Step $5)$.

Player $2m:$ anything $($Step $6)$.

Player $2m+1: 1+1=2$ $($Step $7)$.

Player $2m+2: 1+1=2$ $($Step $8)$.

Player $2m+3: 1+1=2$ $($Step $9)$.

Player $2m+4: 1+1=2$ $($Step $10)$.

Player $2m-1:$ anything $($Step $11)$.

Player $2m:$ anything $($Step $12)$.

$($Note: step $5, 6, 11, 12$ can be anything because they are controlled by team $m.)$

Now we prove this lemma in $2$ cases.\\
\begin{itemize}
    \item[Case 1.] If step $5,6$ are both $2+2=1+3$, then look at steps $1,2,3,4,5,6.$

This contradicts with property \ref{p2}, so team $m$ has no winning strategy.

Similarly, if step $11,12$ are both $2+2=1+3$, then look at steps $7,8,9,10,11,12.$

This contradicts with property \ref{p2}, so team $m$ has no winning strategy.\\

\item[Case 2.] Otherwise, if one of the steps from $5,6$ is not $2+2=1+3$, and one of the steps from $11,12$ is not $2+2=1+3$, then let's do the following after player $2m$'s first move $($continuing after first $12$ steps with $4$ more steps; if $m=3$, the same following process can start from the very first step of player $1)$:

Player $2m+1: 1+1=2$ $($Step $1)$.

Player $2m+2: 1+1=2$ $($Step $2)$.

Player $2m+3: 1+1=2$ $($Step $3)$.

Player $2m+4: 1+1=2$ $($Step $4)$.

Player $2m-1:$ anything $($Step $5)$.

Player $2m:$ anything $($Step $6)$.

Player $2m+1: 1+1=2$ $($Step $7)$.

Player $2m+2: 1+1=2$ $($Step $8)$.

Player $2m+3: 1+1=2$ $($Step $9)$.

Player $2m+4: 1+1=2$ $($Step $10)$.

Player $2m-1:$ anything $($Step $11)$.

Player $2m:$ anything $($Step $12)$.

Player $2m+1: 1+1=2$ $($Step $13)$.

Player $2m+2: 1+1=2$ $($Step $14)$.

Player $2m+3: 2+2=1+3$ $($Step $15)$.

Player $2m+4: 2+2=1+3$ $($Step $16)$.

Note that one of the steps from $5,6$ is not $2+2=1+3$, and one of the steps from $11,12$ is not $2+2=1+3$, so steps $5,6$ will take away at most three ``$2$'s'' in total, and steps $11,12$ will take away at most three ``$2$'s'' in total. Also note that steps $1,2,3,4,5,6,7,8,9,10$ generate eight ``$2$'s'' in total. So after step $12$, there will be at least two $2$'s remaining.

Therefore, for player $2m+1$ and $2m+2$, instead of doing $1+1=2$, they can both do $1+2=3$.

Since team $m$ has winning strategy, either player $2m-1$ or player $2m$ takes the last step.

If player $2m-1$ originally takes the last step, by using the stealing strategy mentioned above, now player $2m-1-2=2m-3$ now takes the last step, which means that player $2m+3$ now takes the last step, so team $m-1$ now has the winning strategy, which contradicts with our assumption.

If player $2m$ originally takes the last step, by using the stealing strategy mentioned above, now player $2m-2$ now takes the last step, which means that player $2m+4$ now takes the last step, so team $m-1$ now has the winning strategy, which contradicts with our assumption.
\end{itemize}

In both cases, we can find contradiction by using stealing strategy, so Lemma \ref{lem4} is proved.
\end{proof}

Therefore, by Lemmas \ref{lem3} and \ref{lem4}, Theorem \ref{thm3} is proved.

\subsection{Proof Of Theorem \ref{thm4}}
Note: in the following proof, since player $6$'s next player is player $1$, player $6$ and player $1$ are considered as two consecutive players.

The $4$-player alliance has $3$ possible cases.\\

\begin{itemize}

\item[Case $1$.] If the $4$-player alliance consists of $4$ consecutive players, then the $2$-player alliance will also be $2$ consecutive players.

To show that the $2$-player alliance does not actually have a winning strategy, the $4$ consecutive players in the big alliance can be regarded as $2$ teams, where each team has $2$ consecutive players.

Therefore, according to Lemma \ref{lem4}, the $2$-player alliance does not have a winning strategy.

Therefore, the $4$-player alliance always has a winning strategy in this case.\\

\item[Case $2$.] If the $4$-player alliance is separated in two parts, and each part has $2$ consecutive players.

Note that this situation is equivalent to the $3$-player game situation, where $2$ of them are in the same team now.

According to Lemma \ref{lem2.1.2}, the single player in the $3$-player game does not have a winning strategy.

Equivalently, the $2$-player alliance does not have a winning strategy in this case.

Therefore, the $4$-player alliance always has a winning strategy in this case.\\

\item[Case $3$.] If the $4$-player alliance is separated in two parts, where one part has $3$ consecutive players and the other part only has $1$ player.

Then, the $2$ players in the $2$-player alliance are separated from each other (if they are not separated, then the $4$ players of the $4$-player alliance will be $4$ consecutive players).

Suppose the $2$-player alliance has a winning strategy, then there always exists a player $q$ from the $2$-player alliance who takes the last step.

Suppose the $3$ consecutive players in the $4$ player alliance are player $a$, $a+1\pmod 6$, $a+2 \pmod 6$. Then let's do the following from player $a$'s first move:

Player $a: 1+1=2$.

Player $a+1: 1+1=2$.

Player $a+2: 2+2=1+3$.

Note that if player $a+1$ does $1+2=3$ instead, then player $q-1$ will now be the player who takes the last step. Since 2 players of the $2$-player alliance are separated, player $q-1$ belongs to the $4$-player alliance. Therefore, the $4$-player alliance now has the winning strategy, which contradicts with our assumption.

Therefore, by using stealing strategy, we proved that the $4$-player alliance has a winning strategy in this case.

\end{itemize}

Thus by Case $1$, Case $2$ and Case $3$, Theorem \ref{thm4} follows.

\subsection{Proof Of Theorem \ref{thm5}}

We look at the case of $7$-player game first.

\begin{lem}\label{lem5}
When $n\geq 32$, if one alliance has $5$ players and the other alliance has 2 players, then the $5$-player alliance always has a winning strategy.
\end{lem}

\begin{proof}
We prove this lemma by considering $2$ cases.\\

\begin{itemize}

\item[Case $1$.] If the $2$ players in the $2$-player alliance are not $2$ consecutive players, then the $5$-player alliance will be separated by $2$ parts (considering player $7$ and player $1$ as two consecutive players).

According to the pigeonhole principle, one of these $2$ parts will contain at least $3$ consecutive players (we call this part ``large part'').

Suppose the $2$-player alliance has a winning strategy, then there exists a player $q$ from the $2$-player alliance who takes the last step.

Suppose the first player in the large part is player $a$, then starting from player $a$'s first move, let's do the following:

Player $a: 1+1=2$.

Player $a+1: 1+1=2$.

Player $a+2: 2+2=1+3$.

Note that instead of doing $1+1=2$, player $a+1$ can do $1+2=3$. Then player $q-1$ is the player who takes the last step of the winning path.

Since the $2$ players in the $2$-player alliance are not consecutive, player $q-1$ belongs to the $5$-player alliance.

As a result, the $5$-player alliance now has the winning strategy, which contradicts our assuption.

Therefore, by using the stealing strategy, we proved that the $5$-player alliance has the winning strategy in this case.\\

\item[Case $2$.] If the $2$ players in the $2$-player alliance are consecutive, then the $5$ players in the $5$-player are also consecutive.

Suppose that the $5$ consecutive players of the $5$-player alliance starts with player $a$.

Then starting with player $a$'s first move, let's do the following:

Player $a: 1+1=2$ $($Step $1)$.

Player $a+1: 1+1=2$ $($Step $2)$.

Player $a+2: 1+1=2$ $($Step $3)$.

Player $a+3: 1+1=2$ $($Step $4)$.

Player $a+4: 1+1=2$ $($Step $5)$.

Player $a+5:$ anything $($Step $6)$.

Player $a+6:$ anything $($Step $7)$.

(Note that player $a+5$ and player $a+6$ are in the $2$-player alliance.)

Player $a: 1+1=2$ $($Step $8)$.

Player $a+1: 1+1=2$ $($Step $9)$.

Player $a+2: 1+1=2$ $($Step $10)$.

Player $a+3: 2+2=1+3$ $($Step $11)$.

Player $a+4: 2+2=1+3$ $($Step $12)$.

Suppose the $2$-player alliance has a winning strategy, then there always exists a player $q$ from the $2$-player alliance who takes the last step.

Note that step $6$ and step $7$ can take away at most four $2$'s in total, and steps $1, 2, 3, 4, 5, 8$ have generated six $2$'s in total.

As a result, after step $8$, there will be at least two $2$'s remaining.

Therefore, player $a+1$ in step $9$ and player $a+2$ in step $10$ can both do $1+2=3$ instead. Then player $q-2$ becomes the player who takes the last step.

Since $2$ players of the 2-player alliance are consecutive, player $q-2$ belongs to the $5$-player alliance. Therefore, the $5$-player alliance now has the winning strategy, which contradicts with our assumption.

Therefore, by using stealing strategy, Case $2$ is proved.
\end{itemize}
Thus by our analysis in Case $1$ and Case $2$, Lemma \ref{lem5} is proved.
\end{proof}

Now let's look at the game of $8$ or more players.

\begin{lem}\label{lem6}
In a $p$-player game with $2$ alliances, when $n$ is significantly large $(n\geq 22)$ and $p\geq 8$, if one alliance has $p-2$ players (let's call it big alliance, which has at least $6$ players), and the other alliance has $2$ players, then the big alliance always has a winning strategy.
\end{lem}

\begin{proof}
We prove this lemma by considering $2$ cases.\\

\begin{itemize}
\item[Case 1.] If the $2$ players of the $2$-player alliance are not consecutive, then the big alliance will be separated by $2$ parts.

Note that the big alliance has at least 6 players. By pigeonhole principle, there will be at least one part having at least $3$ players (let's call it big part).

Suppose $2$-player alliance has a winning strategy. Then for any winning path, there exists a player $q$ in the $2$-player alliance who takes the last step.

Suppose the first player in the big part is player $a$, and let's do the following starting from player $a$'s first move:

Player $a: 1+1=2$ $($Step $1)$.

Player $a+1: 1+1=2$ $($Step $2)$.

Player $a+2: 2+2=1+3$ $($Step $3)$.

Note that instead of doing $1+1=2,$ player $a+1$ can do $1+2=3$ instead in step $2$. Now player $q-1$ becomes the player who takes the last step. Since the $2$ players in the $2$-player alliance are not consecutive, player $q-1$ belongs to the big alliance, so the big alliance now has the winning strategy, which contradicts with our assumption.

Therefore, we proved case $1$ by using stealing strategy.\\

\item[Case 2.] If the $2$ players of the $2$-player alliance are consecutive, then the $p-2$ players of the big alliance are consecutive.

Suppose $2$-player alliance has a winning strategy, then for any winning path, there exists a player $q$ from the $2$-player alliance who takes the last step.

Suppose the big alliance's $p-2$ consecutive players start with player $a$.

$($Note the the big alliance has at least $6$ players, so players $a, a+1, a+2, a+3, a+4, a+5$ are all in the big alliance$).$

Let's do the following starting from player $a$'s first move:

Player $a: 1+1=2$ $($Step $1)$.

Player $a+1: 1+1=2$ $($Step $2)$.

Player $a+2: 1+1=2$ $($Step $3)$.

Player $a+3: 1+1=2$ $($Step $4)$.

Player $a+4: 2+2=1+3$ $($Step $5)$.

Player $a+5: 2+2=1+3$ $($Step $6)$.

Note that player $a+2$ in step $3$ and player $a+3$ in step $4$ can both do $1+2=3$ instead.

If they both do so, then player $q-2$ becomes the player who takes the last step. Since the $2$ players in the $2$-player alliance are consecutive, player $q-2$ belongs to the big alliance. Therefore, the big alliance now has the winning strategy, which contradicts with our assumption.

By using the stealing strategy, we proved case $2$.
\end{itemize}
Thus, by our analysis in Case $1$, and Case $2$, Lemma \ref{lem6} is proved.
\end{proof}

By Lemmas \ref{lem5} and \ref{lem6}, Theorem \ref{thm5} is proved.

\subsection{Proof Of Theorem \ref{thm6}}

Say we have an alliance $a$ with over two-thirds of the players. For a sufficiently large $n$, it can then force the creation of an arbitrary number of $2$'s eventually, as players on the alliance can each produce at least one per round: if each plays $1 + 1 = 2$, opposed players can each remove only two per round by playing $2 + 2 = 1 + 3$, meaning each round, alliance $a$ can net increase the number of $2$'s by at least one.

As such, at some point, alliance $a$ can force the creation of at least $b$ $2$'s. By assumption, alliance $a$ has at least $2b$ subsequent players. Say we have at least $b$ $2$'s and are about to begin the turns of those players.

The first $b$ players could all play $1+1=2$, and the next $b$ players could all play $2+2=1+3$, as the first $b$ players would create $b$ $2$'s and the next $b$ players would use up those and the preexisting $b$ $2$'s. Let us say the turn after this is turn $2b$, changing what we call turn $0$ to make this the case.

Now assume the opposing players have a winning strategy. In that case, after this, there would be a winning strategy from the resultant state for an alliance that goes when the opposed players do. In particular, this alliance goes on turns $2b+p_1, 2b+p_2,$ $2b+p_3, \dots, 2b+p_n$, where $p_1, p_2, \dots$ are some numbers selected to make this the case.

Now, note that the $2b$ players can instead play as follows: the first $b$ can all play $1+2=3$, resulting in the same state as the one we got to after all $2b$ players last time. As such, the same strategy as the one used previously gives the win to the alliance that goes on turns ${b+p_1, b+p_2, b+p_3, \dots, b+p_n}$.

By assumption, alliance $a$ has all players $b$ before a player from the opposition. We know the players with turns $2b+p_1, 2b+p_2, 2b+p_3, \dots, 2b+p_n$ are opposed to alliance $a$, so the players with turns $b+p_1, b+p_2, b+p_3, \dots, b+p_n$ must be on alliance $a$.

As such, this would be a winning strategy for alliance $a$, contradicting the assumption that the opposed players have a winning strategy. As such, alliance $a$ must have a winning strategy.

\subsection{Proof Of Theorem \ref{thm7}}

Assume we have an alliance $a$ with $3b$ consecutive players at some point, and $n \geq 2p+4b$.

First, let's examine the first turn of the $3b$ players. As $n$ is at least $2p+4b$, regardless of when they start, there will be at least enough $1$'s in the game for the first $2b$ of them to play $1+1=2$. If they do so, the next $b$ could all play $2+2=1+3$. Let us say the turn after this is turn $3b$, changing what we call turn $0$ to make this the case.

Now assume the opposing players have a winning strategy. In that case, after this, there would be a winning strategy from the resultant state for an alliance that goes when the opposed players do. In particular, this alliance goes on turns $3b+p_1, 3b+p_2,$ $3b+p_3, \dots, 3b+p_n$, where $p_1, p_2, \dots$ are some numbers selected to make this the case.

Now, note that the $3b$ players can instead play as follows: the first $b$ players can all play $1+1=2$ and the next $b$ players can all play $1+2=3$, resulting in the same state as the one we got to after all $3b$ players last time.

As such, the same strategy as the one used previously gives the win to the alliance that goes on turns $2b+p_1, 2b+p_2, 2b+p_3, \dots, 2b+p_n$. By assumption, alliance $a$ has all players $b$ before a player from the opposition. We know the players with turns $3b+p_1, 3b+p_2, 3b+p_3, \dots, 3b+p_n$ are opposed to alliance $a$, so the players with turns ${2b+p_1, 2b+p_2, 2b+p_3, \dots, 2b+p_n}$ must be on alliance $a$.

As such, this would be a winning strategy for alliance $a$, contradicting the assumption that the opposed players have a winning strategy. As such, alliance $a$ must have a winning strategy.

\subsection{Proof Of Theorem \ref{thm8}}

Let the $2d$ players in the big alliance all do $1+1=2$ for the first $d$ rounds (one round means every player takes a move, and we define the first round starting from the big alliance's first move).

If in any of the first $d$ round, the $d$ consecutive players in the small alliance all do $2+2=1+3$, then we can directly let the second half of the $2d$ players in the big alliance (which is $d$ consecutive players) all do $1+2=3$ instead in the next round.

In this case, suppose that small alliance has a winning strategy. Then for any winning path, there exists a player $q$ from the small alliance who takes the last step.

Note that player $q$ belongs to small alliance, so player $q-d$ belongs to big alliance. Since player $q$'s last winning step becomes player $q-d$'s last step, so the big alliance now has the winning strategy, which contradicts with our assumption.

Otherwise, if in each of the first $d$ rounds, there is at least one player from small alliance who does not take $2+2=1+3$, which means that this step will take away at most one $2$. Then there will be at least one $2$ generated in each round.

As a result, after $d$ rounds, there are at least $d$ $2$'s generated for the stealing. After that, in the $(d+1)$th round, we can let the first half of the players in big alliance (which are the first $d$ consecutive players) all do $1+1=2$, and the second half of the big alliance (which are last $d$ consecutive players) all do $2+2=1+3$.

In this case, suppose that the small alliance has a winning strategy. Then for any winning path, there exists a player $q$ from the small alliance who takes the last step.

Note that player $q$ belongs to the small alliance, so player $q-d$ belongs to big alliance. In this round, the first half of the big alliance (which are first the $d$ consecutive players) can all do $1+2=3$ instead, so player $q$'s last winning step becomes player $q-d$'s last step. As a result, the big alliance now has the winning strategy, which contradicts with our assumption.

Therefore, the big alliance always has a winning strategy, so we have proved this theorem.\\

\ \\

\end{document}